\documentclass[a4paper,12pt, reqno]{amsart}
\usepackage{amsmath,amssymb,enumerate,epsfig,xcolor}

\usepackage{graphicx}
\usepackage{import}
\usepackage{xifthen}
\usepackage{pdfpages}
\usepackage{transparent}

\usepackage{enumitem}
\usepackage{comment}

\usepackage{float}

\newtheorem{thm}{Theorem}

\newtheorem{prop}[thm]{Proposition}
\newtheorem{lemma}[thm]{Lemma}
\newtheorem{claim}{Claim}
\newtheorem*{cor*}{Corollary}

\newtheorem{q}[thm]{Question}

\newtheorem{Rmks}[thm]{Remarks}

\theoremstyle{definition}
\newtheorem{definition}[thm]{Definition}
\newtheorem*{definition*}{Definition}
\newtheorem{exple}[thm]{Example}
\newtheorem*{exple*}{Example}
\newtheorem{remark}[thm]{Remark}
\newtheorem*{rmk}{Remark}

\newtheorem{exples}[thm]{Examples}
\newtheorem*{exples*}{Examples}

\newcommand\xqed[1]{%
	\leavevmode\unskip\penalty9999 \hbox{}\nobreak\hfill
	\quad\hbox{#1}}
\newcommand\demo{\xqed{$\Diamond$}}

\newcommand{\N}{\mathbb{N}}
\newcommand{\Z}{\mathbb{Z}}
\newcommand{\R}{\mathbb{R}}
\newcommand\Om{\Omega}
\newcommand\om{\omega}

\renewcommand\phi{\varphi}

\newcommand\omst{\omega_{\mathrm{st}}}
\newcommand{\Forms}[2]{\Omega^{#1, #2}}
\newcommand{\cat}{\mathcal{C}}

\newcommand{\mcat}{\mathcal{M}}
\newcommand{\mcattilde}{\widetilde{\mathcal{M}}}
\newcommand{\wtilde}[1]{\widetilde{#1}}

\newcommand{\partition}[1]{\mathcal{P}^{#1}}
\newcommand{\cathat}{\widehat{\mathcal{C}}}
\newcommand{\obj}{\mathcal{O}}
\newcommand{\Int}{\mathrm{Int}}

\newcommand{\vol}{\mathrm{Vol}}

\begin{document}
	
\title[Non-target-representable symplectic capacities]{Examples of non-target-representable symplectic capacities}
\author{Yann Guggisberg\and Fabian Ziltener}
\address{Affiliation of Y.~Guggisberg: Utrecht University\\
	mathematics institute\\
	Hans Freudenthalgebouw\\
	Budapestlaan 6\\
	3584 CD Utrecht\\
	The Netherlands
}
\email{y.b.guggisberg@uu.nl}
\address{Affiliation of F.~Ziltener: Utrecht University\\
	mathematics institute\\
	Hans Freudenthalgebouw\\
	Budapestlaan 6\\
	3584 CD Utrecht\\
	The Netherlands
}
\email{f.ziltener@uu.nl}
\thanks{Y.~Guggisberg's work on this publication is part of the project \emph{Symplectic capacities, recognition, discontinuity, and helicity} (project number 613.009.140) of the research programme \emph{Mathematics Clusters}, which is financed by the Dutch Research Council (NWO). We gratefully acknowledge this funding.}

\setcounter{tocdepth}{1}
	
\maketitle

\begin{abstract}
	We give the first concrete examples of symplectic capacities that are not target-representable. This provides some answers to a question by Cieliebak, Hofer, Latschev, and Schlenk.
\end{abstract}

\tableofcontents

\section{Main result and simpler but less interesting examples}
\subsection{Introduction and main result}

The results of this article are concerned with non-target-representable capacities on small weak differential form categories. These generalize symplectic capacities on small symplectic categories. To explain our results, we define the notion of a capacity in the following even more general setup.

Let $\mathcal{C}$ be a small category. We denote by $\obj$ its set of objects. We denote $\R^+:=(0, \infty)$ and fix an $\R^+$-action\footnote{An action of a group $G$ on a small category $\cat$ is a map $\rho:G\times\cat\to\cat$, such that for every $g, h\in G, \rho_g$ is a functor, $\rho_{gh}=\rho_g\rho_h$, and $\rho_e$ is the identity functor, where $e$ denotes the neutral element of $G$.} on $\mathcal{C}$.

\begin{definition}[generalized capacity]\label{def_capacity}
	A \emph{(generalized) capacity}\footnote{In \cite{CHLS07}, a \emph{capacity} is defined to be a generalized capacity that is non-trivial. Since we do not use this condition in this article, we will use the word \emph{capacity} in the sense of \emph{generalized capacity}.} on $\mathcal{C}$ is a map $c:\mathcal{O} \to [0, \infty]$ with the following properties:
	\begin{enumerate}[label = (\roman*)]
		\item \textbf{(monotonicity)} If  $A$ and $B$ are two objects in $\mathcal{O}$ such that there exists a $\mathcal{C}$-morphism from $A$ to $B$, then
		\begin{equation*}
			c(A) \leq c(B).
		\end{equation*}
		\item \textbf{(conformality)} For every $A \in \mathcal{O}$ and every $a \in \R^+$ we have
		\begin{equation*}
			c(aA) = ac(A),\;\footnote{Here, $aA$ denotes the action of $a$ on the object $A$.}
		\end{equation*}
	\end{enumerate}
\end{definition}

To define the notion of an embedding capacity, let $\cathat$ be a category determined by a formula. By this we mean a category whose objects, morphisms, and pairs\footnote{By this, we mean pairs $((f, g), h)$, where $(f, g)$ is the input of the composition map and $h$ is the output, i.e. $f \circ g = h$.} constituting the composition map are given by well-formed logical formulas.

\begin{remark}\label{rmk_zfc}
	This article is based on the Zermelo-Fraenkel axiomatic system with choice (ZFC). In most cases, the collection of objects and morphisms of a category are proper classes and not sets. Therefore, we need to be careful when talking about general categories. However, categories determined by a formula can be dealt with in ZFC. 
	\demo
\end{remark}

We fix an $\R^+$-action on $\widehat{\mathcal{C}}$ that is defined by a logical formula as in \cite[Appendix A, Definition 18]{GZ22}. Let $\cat$ be a small $\R^+$-invariant subcategory of $\cathat$. We denote by $\obj$ its set of objects. \label{setting}

\begin{definition}[embedding capacity]\label{def_embedding_capacity}
	For every object $A$ of $\cathat$, we define the \emph{domain-embedding capacity} for $A$ on $\cat$ to be the map
	\begin{gather*}
		c_A:=c_A^{\obj, \cathat}: \obj \to [0, \infty]\\
		c_A(B):=\sup\left\{a\in(0, \infty)\,\big|\, \exists \,\cathat\text{-morphism } aA \to B\right\}.
	\end{gather*}
	Similarly, for every object $B$ of $\cathat$, we define the \emph{target-embedding capacity} for $B$ on $\cat$ to be the map
	\begin{gather}
		c^B:=c^B_{\obj, \cathat}: \obj \to [0, \infty]\nonumber\\
		c^B(A):=\inf\left\{b\in (0, \infty)\,\big|\, \exists \, \cathat\text{-morphism } A \to bB\right\}.\label{def_target_embedding_cap}
	\end{gather}
	Here we use the conventions $\sup\emptyset = 0$ and $\inf \emptyset=+\infty$.
	
	The maps $c_A$ and $c^B$ are capacities on $\cat$.
\end{definition}

In \cite{CHLS07}, K. Cieliebak, H. Hofer, J. Latschev, and F. Schlenk defined a \emph{symplectic category} to be a subcategory of the category of symplectic manifolds with symplectic embeddings as morphisms, such that for every object $(M, \omega)$ of the subcategory and every $a\in (0, \infty)$ the pair $(M, a\omega)$ is also an object of the subcategory. \footnote{As mentioned in Remark \ref{rmk_zfc}, arbitrary subcategories cannot be dealt with in ZFC. This can be resolved by considering isomorphism-closed categories as defined in \cite[Definition 16]{GZ22}.} In \cite[Problem 1]{CHLS07}, these authors asked the following question for symplectic categories.
\begin{q}\label{question_rep}
	Which capacities on a symplectic category can be represented as $c^{(X, \Om)}$ for a connected symplectic manifold $(X, \Om)$?
\end{q}
\begin{rmk}
	Let $\cat$ be the category whose objects are the connected open sets of $\R^{2n}$ and whose morphisms are the symplectic embeddings. Then, as explained in \cite[Example 2]{CHLS07}, every capacity on $\cat$ can be represented by a possibly uncountable disjoint union $(X, \Om)$ of manifolds. If this union is uncountable, then $(X, \Omega)$ itself is not a manifold, since its topology is not second countable.
	\demo
\end{rmk}

For $n\geq 2$, the set of capacities on the category of symplectic manifolds of dimension $2n$ and symplectic embeddings that are representable by $c^{(X, \Omega)}$ for some symplectic manifold $(X, \Omega)$ has cardinality at most $\beth_1$, see \cite[Corollary 58]{JZ21}. On the other hand, in \cite[Theorem 17]{JZ21}, D. Joksimovi\'{c} and F. Ziltener showed that the cardinality of the set of all capacities on this category is $\beth_2$. This means that almost no capacity is target-representable, even if we allow for disconnected targets. See \cite[theorem on page 8 and Corollary 20]{JZ21}.\footnote{In that article, the authors used a different setup for capacities. Namely, they defined this notion on isomorphism-closed (as defined in Footnote \ref{footnote_isoclosed}) subcategories of the universal form category $\Forms{m}{k}$ (as in Definition \ref{def_univeral_from_cat}). Such a category is not small, except if it is empty. The two theories of capacities are equivalent, if we only consider small weak form categories (as in Definition \ref{def_small_weak_cat}), such that for every pair of objects $A,B$, every $\Forms{m}{k}$-isomorphism from $A$ to $B$ is a morphism in the small weak form category.} In particular, non-target-representable capacities exist.

The present article provides examples of capacities that are not target-representable in the following more general setting.

Let $m, k \in \N_0:=\{0, 1, 2, \dots\}$.
\begin{definition}[universal form category]\label{def_univeral_from_cat}
	We define the \emph{universal form category} \[\Forms{m}{k} \] as follows:
	\begin{itemize}
		\item Its objects are the pairs $(M, \omega)$, where  $M$ is a (smooth) manifold\footnote{All the manifolds in this article are smooth, finite-dimensional, and are allowed to have boundary.} of dimension $m$ and $\omega$ is a differential $k$-form on $M$.
		\item Its morphisms are (smooth) embeddings\footnote{We do not impose any condition involving the boundaries of the manifolds, here.} that intertwine the differential forms.
	\end{itemize}
	For two objects $(M, \omega)$ and $(M', \omega')$, we denote \[(M, \omega) \hookrightarrow (M', \omega')\] iff there exists a morphism between them. We equip $\Forms{m}{k}$ with the following $\R^+$-action. For every $a\in \R$, we define $a\cdot$ to be the functor from $\Forms{m}{k}$ to itself that sends every object $M:=(M, \omega)$ to
\[aM:=(M, a\omega)\]
and every morphism $\varphi:(M, \omega) \hookrightarrow (M', \omega')$ to the same map viewed as a morphism $(M, a\omega) \hookrightarrow (M', a\omega')$.
\end{definition}

\begin{rmk}
	The category $\Forms{m}{k}$ is determined by a logical formula.
	\demo
\end{rmk}

\begin{definition}[small weak $(m, k)$-form category]\label{def_small_weak_cat} A \emph{small weak}\footnote{Here, the word ``weak'' refers to the fact that the subcategory need not be isomorphism-closed. A subcategory $\cat'$ of $\cat$ is isomorphism-closed iff every isomorphism in $\cat$ starting at some object of $\cat'$ is also a morphism  in $\cat'$.\label{footnote_isoclosed}} $(m, k)$-\emph{(differential) form category} is a small $\R^+$-invariant subcategory of $\Forms{m}{k}$.
\end{definition}

The present article provides examples of capacities that are not (closedly) target-representable in the following sense. 
\begin{definition}[(closedly) target-representable capacity]\label{def_target_representable}
	Let $\cat$ be a small weak $(m, k)$-form category. We denote its set of objects by $\obj$. Let $c$ be a capacity on $\cat$. We call $c$ \emph{target-representable}  iff there exists an object $(X, \Omega)$ of $\cathat:=\Forms{m}{k}$, such that $c= c_{\obj, \cathat}^{(X, \Omega)}$ (defined as in \eqref{def_target_embedding_cap}). We call $c$ \emph{closedly target-representable} iff there is such an $(X, \Omega)$ with $d\Omega=0$.
\end{definition}

To state our main result, we need the following notions. Let $k, n \in \N_0$, such that $n\geq 2$.

\begin{definition}[helicity]\label{def_helicity}
	Let $N$ be a closed\footnote{This means compact and without boundary.} $(kn -1)$-dimensional manifold, $O$ an orientation on $N$, and $\sigma$ be an exact $k$-form on $N$. We define the helicity of $(N, O, \sigma)$ to be the integral
	\begin{equation*}
		h(N, O, \sigma) := \int_{N, O} \alpha \wedge \sigma^{\wedge(n-1)},
	\end{equation*}
	where $\alpha$ is an arbitrary primitive of $\sigma$ and $\int_{N, O}$ denotes the integral on $N$ with respect to the orientation $O$. By Lemma 21 in \cite{GZ22} this number is well-defined.
\end{definition}

\begin{Rmks}[helicity]\label{rmk_helicity}
		(See \cite[Remark 31]{JZ21}.)
		\begin{itemize}
			\item Helicity vanishes when $k$ is odd.
			\item Helicity is not well-defined for $n=1$.
			\demo
		\end{itemize}
	
\end{Rmks}

Let $M$ be a $kn$-dimensional (smooth)  manifold with boundary, $O$ an orientation on $M$, and $\omega$ an exact $k$-form on $M$. We define\[I_M:=\left\{\text{connected components of } \partial M\right\}.\]For every $i \in I_M$, we denote by $O_i$ the orientation induced on $i$ and by $\omega_i$ the pullback of $\omega$ to $i$.

Assume that $M$ is compact.
\begin{definition*}
	We define the \emph{boundary helicity} of $(M, O, \omega)$ to be the map $h_M := h_{M, O, \omega}: I_M \to \R$ given by
	\begin{equation*}
		h_{M, O, \omega} (i) := h(i, O_i, \omega_i).
	\end{equation*}
\end{definition*}

\begin{definition*}[maxipotency]
	Assume that $k, n \geq 1$. A $k$-form $\om$ on a $kn$-dimensional manifold $M$ is called \emph{maxipotent} iff $\omega^{\wedge n}$ is nowhere-vanishing. We call an object $(M, \omega)$ of $\Forms{kn}{k}$  maxipotent iff $\omega$ is \mbox{maxi}potent.
\end{definition*}

\begin{rmk}[maxipotent]
	\begin{enumerate}[label = (\roman*)]
		\item\label{rmk_maxipotent_odd} If $n \geq 2$ and $k$ is odd, we have that $\omega \wedge \omega =0$, and hence there are no maxipotent $k$-forms.
		\item\label{rmk_maxipotent_orientation} A maxipotent $k$-form $\omega$ on a $kn$-dimensional manifold induces an orientation on that manifold. The orientation is given by the $kn$-form $\omega^{\wedge n}$, which is nowhere-vanishing, and hence a volume form.
		\demo
	\end{enumerate}
\end{rmk}

Let $k, n\in \N:=\{1, 2, 3, \dots\}$, such that $n\geq 2$. Let $\mcat$ be an uncountable set of compact, 1-connected\footnote{This means connected and simply connected.}, exact, and maxipotent objects of $\Forms{kn}{k}$ with induced volume 1. Assume that every element of $\mcat$ has exactly three boundary components, one with helicity 4, and the other two with  helicity $\leq -1$. Moreover, assume that no two elements of $\mcat$ are isomorphic.

\begin{rmk}
	By Remark \ref{rmk_helicity}, if $k$ is odd, helicity vanishes. Since $\mcat$ is a non-empty set of objects with non-zero boundary helicity, it follows that $k$ is even.
	\demo
\end{rmk}

For every set $X$, we denote by $\mathcal{P}(X)$ its power set. We define the set
\begin{align*}
		\obj_0^{m, k}:=\big\{&(M, \omega) \text{  object of }\Forms{m}{k}\,\big|\,\\
		&\text{The set underlying }M\text{ is a subset of }\mathcal{P}(\N_0).\big\}.\footnotemark
\end{align*}\footnotetext{The reason for requiring the set underlying $M$ to be included in $\mathcal{P}(\N_0)$ is to ensure that $\Forms{m}{k}_0$ is a set. Here we use the ZF axiom of restricted comprehension.}

We denote by $\mcattilde$ the set of compact, 1-connected, exact, maxipotent elements of $\obj_0^{kn, k}$  of volume (strictly) greater than 1, that have exactly three boundary components, one with helicity at least 4, and two with negative helicity adding up to at least $-3$.

\begin{thm}[non-target-representable capacity]\label{thm_sets_fulfill_conditions}
	Let $\cat$ be a small weak $(kn, k)$-form category. We denote by $\obj$ its set of objects. Assume that $\obj$ includes $\mcat\cup \mcattilde$. Then the capacity \[c_{\mcattilde}:=\sup_{\wtilde{M}\in\mcattilde}c_{\wtilde{M}}^{\obj, \Forms{kn}{k}}\] is not closedly target-representable (as in Definition \ref{def_target_representable}).
\end{thm}
For a proof, see page \pageref{proof_thm_conditions}.

\begin{rmk}
	For every $\mcat$ as above, the full subcategory of $\Forms{kn}{k}$ consisting of all positive rescalings of elements of $\mcat\cup\mcattilde$ satisfies the hypotheses of Theorem \ref{thm_sets_fulfill_conditions}.
	\demo
\end{rmk}
\begin{definition*}
	 A \emph{small weak symplectic category} is a small weak $(2n, 2)$-form category whose objects are symplectic manifolds.
\end{definition*}

Theorem \ref{thm_sets_fulfill_conditions} has the following immediate consequence.
\begin{cor*}
	Let $\cat$ be a small weak symplectic category whose set of objects includes $\mcat\cup\mcattilde$. Then $c_{\mcattilde}$ is not equal to $c^{(X, \Omega)}$ for any (even disconnected) symplectic manifold $(X, \Omega)$.
\end{cor*}
This provides some answer to Question \ref{question_rep}. To our knowledge, this is the first result concerning this question, apart from \cite[theorem on page 8]{JZ21}, which states that almost no normalized symplectic capacity on the category of all symplectic manifolds is domain or target-representable.

\begin{exple}[set $\mcat$ as in Theorem \ref{thm_sets_fulfill_conditions}]\label{exple_rounded_rectangle}
	Let $n\geq 2$. We equip $\R^{2n}$ with the standard symplectic form $\omst$. We give an example of a set $\mcat$ as above, whose elements are submanifolds of $\R^{2n}$. We choose a closed $2n$-dimensional ``rounded rectangle''\footnote{By this we mean a rectangle with rounded corners (of any dimension).} $R$ in $\R^{2n}$ whose volume with respect to $\omst$ is 4. For every $a \in \left[1, \frac{3}{2}\right]$, we choose two closed  ``rounded rectangles'' $R_{0, a}$ and $R_{1,a}$  of volume $a$ and $3-a$, respectively, such that
	\begin{equation*}
		R_{0, a}\cap R_{1,a} = \emptyset \quad \text{and} \quad R_{0, a}\cup R_{1, a} \subseteq \Int(R),
	\end{equation*}
	where $\Int(X)$ denotes the interior of a manifold $X$. We define
	\begin{equation*}
		M_a := R \setminus \big(\Int(R_{0, a})\cup \Int(R_{1, a})\big).
	\end{equation*}
	(See Figure \ref{fig_rectangle} for an illustration.) We define
	\begin{equation*}
		\mcat := \left\{(M_a, \omst|_{M_a})\,\Big|\, a \in \left[1, \frac{3}{2}\right]\right\}.
	\end{equation*}
	Then $\mcat$ is an uncountable set of compact, 1-connected, exact objects of $\Forms{2n}{2}$ with volume 1. By Stokes's theorem for helicity (see Lemma \ref{lemma_stokes_helicity}), for every $a\in\left[1, \frac{3}{2}\right]$, the boundary components of $M_a$ have helicity $4, -a$, and $-3+a$, respectively. If $a$ and $a'$ are such that $M_a$ and $M_{a'}$ are isomorphic, then $\partial M_a$ and $\partial M_{a'}$ are isomorphic as presymplectic manifolds and hence have the same helicities. It follows that $a=a'$. Hence $\mcat$ has the desired properties.
	\begin{figure}[ht]
		\centering
		
	\def\svgwidth{0.7\columnwidth}
	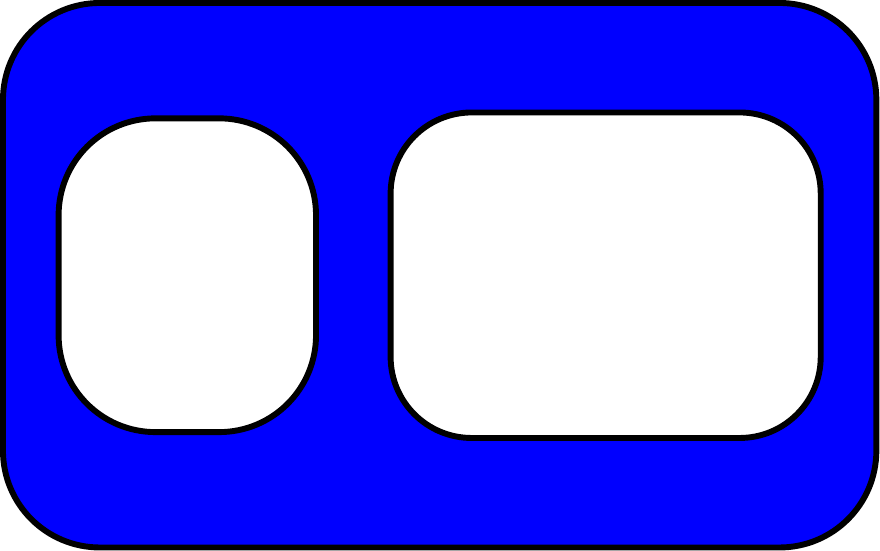

		\caption{A manifold $M_a$ as in Example \ref{exple_rounded_rectangle}. The outside ``rounded rectangle'' has a volume of 4 while the two inner holes have a combined volume of 3.}
		\label{fig_rectangle}
	\end{figure}
\end{exple}

The idea of the proof of Theorem \ref{thm_sets_fulfill_conditions} is the following. Let $X$ be an object of $\Forms{m}{k}$. Without loss of generality, we may assume that $c^X$ is finite on $\mcat$. Since $\mcat$ is uncountable, $X$ has countably many connected components, and the elements of $\mcat$ are connected, it follows that two elements $M$ and $M'$ of $\mcat$ embed\footnote{For two objects $X_0, X_1$ of $\Forms{m}{k}$, we say that $X_0$ embeds into $X_1$ if there exists an $\Forms{m}{k}$-morphism from $X_0$ to $X_1$.} into the same connected component of $X$ after suitably rescaling them. Since no element of $\mcat$ is isomorphic to another one, one of these embeddings is non-surjective. Without loss of generality, assume that $M$ is the domain of this non-surjective embedding. By adding a bulge to the image of $M$, we construct an element $\wtilde{M_0}$ of $\mcattilde$ that embeds into $X$ with the same rescaling factor as $M$. Figure \ref{fig_bulge} illustrates this construction. Stokes' theorem for helicity (Lemma \ref{lemma_stokes_helicity}) implies that $c_{\mcattilde}(M)$ is small for every $M \in \mcat$. Comparing the values of $c_{\mcattilde}$ and $c^X$ on $\wtilde{M_0}$ and the elements of $\mcat$, it follows that $c_{\mcattilde}$ and $c^X$ are not equal.

\begin{figure}[ht]
	\centering
	
	\def\svgwidth{0.8\columnwidth}
	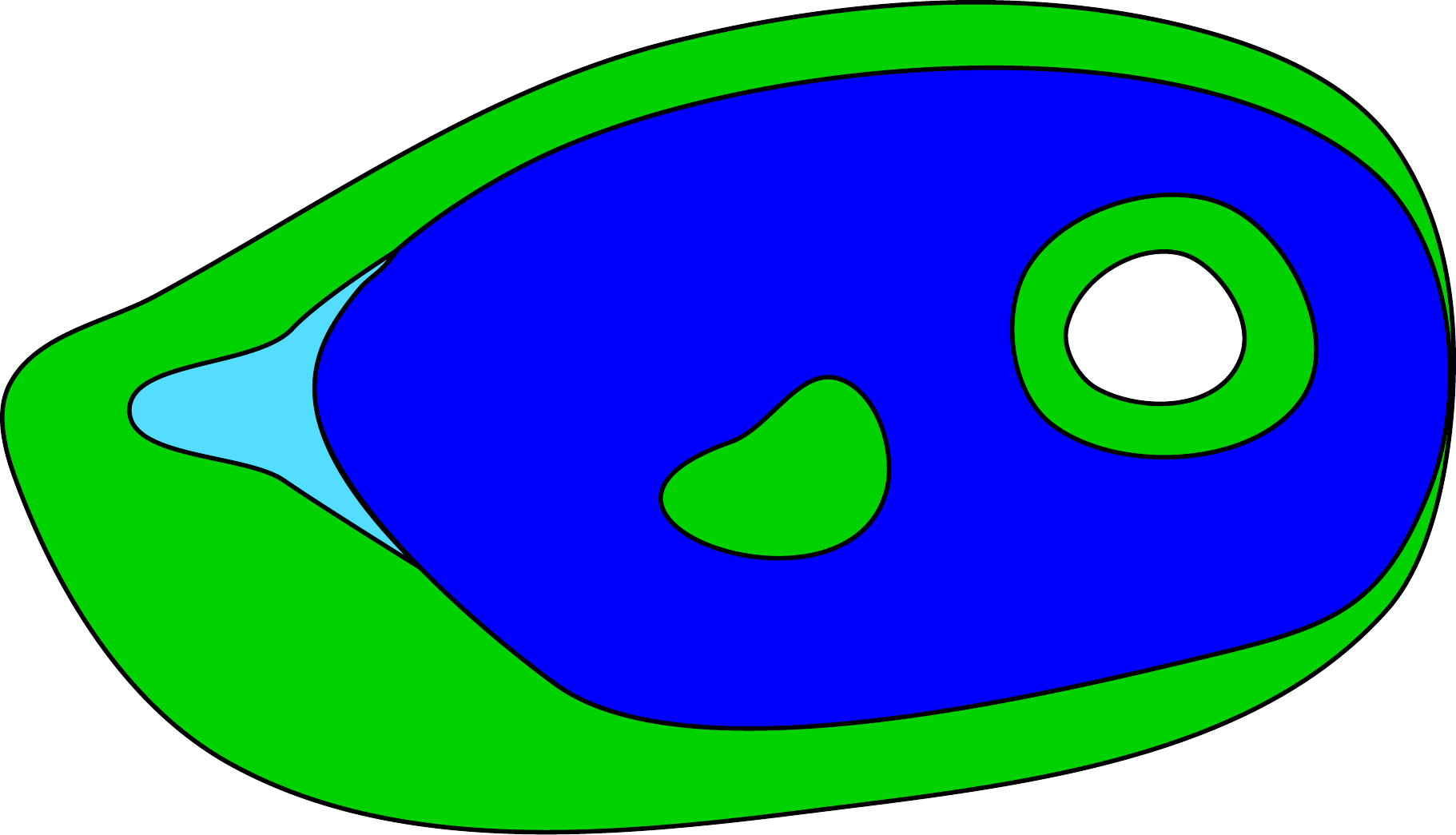

	\caption{The construction of $\wtilde{M}$ in the proof of Theorem \ref{thm_sets_fulfill_conditions}. The dark blue manifold is $\varphi(M)$. The light blue bulge is  added to construct $\wtilde{M}$. The green manifold is $Y$.}
	\label{fig_bulge}
\end{figure}

\subsection{Simpler but less interesting examples of non-target-representable capacities}

Our next two results provide examples of non-target-representable capacities that are simpler than the capacity in Theorem \ref{thm_sets_fulfill_conditions}, but defined on categories including closed manifolds as objects.

Let $k, n \in \N$.
\begin{prop}[non-target-representable capacity on closed manifolds]\label{prop_disjoint_union_nonrepresentable}
	Let $M:=(M, \omega)$ be a non-empty closed connected object of $\Forms{kn}{k}$, such that $\omega$ is maxipotent. Let $c$ be a capacity on some small weak $(kn,k)$-form category whose objects include all disjoint unions of two (positively) rescaled copies of $M$, such that $c$ is finite on each such union. Then $c$ is not target-representable.
\end{prop}
For a proof, see page \pageref{proof_disjoint_union}.
\begin{exples}\label{exple_prop}
	The following capacities satisfy the hypotheses of Proposition \ref{prop_disjoint_union_nonrepresentable} and are therefore not target-representable on any category as in this Proposition:
	\begin{itemize}
		\item $c=c_Y$ for any non-empty object $Y$ of $\Forms{kn}{k}$.
		\item The volume capacity, if the given category consists of maxipotent objects.
	\end{itemize}
\end{exples}

The idea of the proof of Proposition \ref{prop_disjoint_union_nonrepresentable} is the following. Assume that the capacity $c$ is target-represented by an object $X$. Since $c(aM\sqcup M)$ is finite for every $a\in (0, \infty)$, it follows that a rescaled copy of every such disjoint union imbeds into $X$. This implies that there is an uncountable set of rescaled copies of $M$ that embed into $X$. Since $M$ is closed, it follows that the image of each copy of $M$ under such an embedding is a connected component of $X$. Since every rescaled copy of $M$ has a different volume, it follows that no two rescaled copies get mapped to the same connected component. This implies that $X$ has uncountably many connected components. Therefore $X$ is not a (second countable) manifold. The statement follows.

Let $k \geq 1$. The next result gives an example of a non-target-representable capacity on a weak small $(2k, k)$-form category that consists of connected objects, which is not the case in Proposition \ref{prop_product_of_surfaces}.

\begin{prop}[non-target-representable capacity on closed connected manifolds]\label{prop_product_of_surfaces}
	Let $M_0$ and $M_1$ be non-empty closed connected $k$-dimensional manifolds, $\sigma_0$ and $\sigma_1$ be volume forms of volume 1 on $M_0$ and $M_1$, respectively, and $c$ be a capacity on some small weak $(2k, k)$-form category whose objects include the products $\big(M_0\times M_1, a_0 \sigma_0\oplus a_1\sigma_1\big)$ for all $a_0, a_1 \in (0, \infty)$, such that $c$ is finite on each such product. Then $c$ is not target-representable.
\end{prop}
For a proof, see page \pageref{proof_product}.

The idea of the proof of Proposition \ref{prop_product_of_surfaces} is the following. It follows from the Degree Theorem that there exist uncountably many products of the form $\big(M_0\times M_1, a_0 \sigma_0\oplus a_1\sigma_1\big)$ that are not isomorphic, even after rescaling. Since $c$ is finite on these products, it follows as in the proof of Proposition \ref{prop_disjoint_union_nonrepresentable}, that any $X$ that target-represents $c$, has uncountably many connected components.

\begin{exples*}
	The capacities in Example \ref{exple_prop} also fulfill the hypotheses of Proposition \ref{prop_product_of_surfaces} and hence are not target-representable on any category as in this proposition.
\end{exples*}

The following remark provides an example of non-target-representable capacity that is even simpler, but defined on a non-full subcategory of $\Forms{2}{2}$.
\begin{rmk}[non-target-representable capacity on the restriction category of bounded open subsets of $\R^2$]
	We denote by $\omst$ the standard symplectic form on $\R^2$. We define a category $\cat$ whose objects are the bounded open subsets of $\R^2$ equipped with $\omst$ and whose morphisms between two objects $U$ and $U'$ are given by the set \[\left\{\varphi|_U\,\big|\, \varphi \text{ is a symplectomorphism of }\R^2 \text{ and }\varphi(U)\subseteq U'\right\}.\] We define the \emph{bounded hull} of an object $U$ of $\cat$ to be the union of $U$ and the bounded components of $\R^2\setminus U$. For every object $U$ of $\cat$, we define $c(U)$ to be the area (Lebesgue measure) of the bounded hull of $U$. Since the bounded hull construction is monotone, $c$ is a capacity on $\mathcal{C}$. For $a>0$, we denote by $B(a)$ ($\overline{B}(a)$) the open (closed) ball of area $a$ around 0 in $\R^2$. The sets $B(2)\setminus\overline{B}(1)$ and $B(1)\setminus\{0\}$ are isomorphic in $\Forms{2}{2}$. This implies that for every object $X$ of $\Forms{2}{2}$, we have that $c^X(B(2)\setminus\overline{B}(1))=c^X(B(1)\setminus\{0\})$. However, we have that $c(B(2)\setminus\overline{B}(1))=2$ and $c(B(1)\setminus\{0\})=1$. Hence $c$ is not target-representable.
	\demo
\end{rmk}

Here is a silly example of non-target-representable capacity.
\begin{rmk}[non-target-representable capacity on a discrete category]
	Let $A$ and $B$ be two isomorphic objects of $\Forms{m}{k}$. We consider the small weak $(m, k)$-form category whose objects are rescaled copies of $A$ and $B$ and whose morphisms are the identity morphisms. Then any $\R^+$-equivariant function $c$ on this small category is a capacity. However,  if $c(A)\neq c(B)$ then $c$ is not target-representable, since $c^{(X, \Omega)}(A)=c^{(X, \Omega)}(B)$ for every object $(X, \Omega)$ of $\Forms{m}{k}$.
	\demo
\end{rmk}

\subsection*{Organization of the article}
In Section \ref{section_proofs}, we prove Theorem \ref{thm_sets_fulfill_conditions} and Propositions \ref{prop_disjoint_union_nonrepresentable} and \ref{prop_product_of_surfaces}. Appendix \ref{section_appendix} deals with the general Cartan formula and Stokes' Theorem for helicity.

\section{Proofs of the results}\label{section_proofs}

For the proof of Theorem \ref{thm_sets_fulfill_conditions}, we need the following. For two objects $(M, \omega)$ and $(M', \omega')$ of $\Forms{m}{k}$, we write \[(M, \omega)\cong(M', \omega')\] iff the two manifolds are isomorphic. We write \[(M, \omega)\sim (M', \omega')\] iff there exists a number $a\in(0, \infty)$, such that $(M, a\omega)\cong (M', \omega')$. The main ingredient of the proof of Theorem \ref{thm_sets_fulfill_conditions} is the following proposition.
	
\begin{prop}[criteria for non-representability]\label{prop_non_representable}
	Let $k,m \in \N_0$ and $\mcat, \mcattilde$ be sets consisting of objects of $\Forms{m}{k}$ with the following properties:
	\begin{enumerate}[label = (\alph*)]
		\item\label{hyp_conn} The elements of $\mcat$ are connected.
		\item\label{hyp_uncountable} $\mcat$ is uncountable.
		\item\label{hyp_no_rescaled_copy} $\forall M, M' \in\mcat$: if $M\sim M'$, then $M=M'$.
		\item\label{hyp_small_sup} $\sup\left\{c_{\wtilde{M}}(M)\,\Big|\,M\in\mcat, \wtilde{M}\in\mcattilde\right\}<1.$
		\item\label{hyp_nonsurjective_embedding} For every connected object $Y$ of $\Forms{m}{k}$ whose second component is a closed differential form the following holds. If there exists an element $M\in\mcat$ and a nonsurjective morphism $M\hookrightarrow Y$, then there exists $\wtilde{M}\in\mcattilde$ and a morphism $\wtilde{M} \hookrightarrow Y$.
	\end{enumerate}
Then the capacity
	\begin{equation*}
		c_{\mcattilde}:=\sup_{\wtilde{M}\in\mcattilde}c_{\wtilde{M}}
	\end{equation*}
	is not closedly target-representable.
\end{prop}

For a proof, see page \pageref{proof_prop_non_representable}.

\begin{remark}[criteria for non-representability]\label{rmk_weaker_assumption}
	If condition \ref{hyp_nonsurjective_embedding} holds for \emph{every} connected object $Y$ (not only those with closed differential forms), then the same proof as that of Proposition \ref{prop_non_representable} shows that $c_{\mcattilde}$ is not target-representable. However, we will not need this version of the proposition.
	\demo
\end{remark}

\begin{rmk}
	If $m=0$, then $k=0$ by conditions \ref{hyp_conn} and \ref{hyp_uncountable}. Condition \ref{hyp_no_rescaled_copy} then implies that $\mcat$ has at most three elements, which contradicts condition \ref{hyp_uncountable}. It follows that the statement of Proposition \ref{prop_non_representable} is void for $m=0$.
	\demo
\end{rmk}

\begin{remark}\label{rmk_a_0}
Hypothesis \ref{hyp_small_sup} implies that for every manifold $X$ that target-represents $c_{\mcattilde}$, we have
\begin{equation*}
	\sup_{M\in\mcat}c^X(M) = \sup\left\{c_{\wtilde{M}}(M)\,\big|\,M\in\mcat, \wtilde{M}\in\mcattilde\right\}<1.
\end{equation*}
\demo
\end{remark}

Let $\varphi: M \to M'$ be an embedding between two topological manifolds. We denote
\begin{align*}
	&I := I_M := \{\text{connected components of }\partial M\},\\
	&I' := I_{M'},\\
	& P := M' \setminus \varphi(\mathrm{Int}(M)).
\end{align*}

We denote by
\begin{equation*}
	\psi: I \sqcup I' \to \mathcal{P}(P)
\end{equation*}
the map induced by $\varphi$ on $I$ and given by the inclusion on $I'$. We define a partition
\begin{equation}\label{eqn_partiton}
	\mathcal{P}^{\varphi}
\end{equation}
on $I \sqcup I'$ by declaring that two boundary components $i, j \in  I \sqcup I'$ lie in the same element of $\mathcal{P}^{\varphi}$ if and only if there is a continuous path in $P$ that starts in $\psi(i)$ and ends in $\psi(j)$.

The following three lemmas will also be used in the proof of Theorem \ref{thm_sets_fulfill_conditions}.

\begin{lemma}[partition induced by an embedding]\label{lemma_property_partition}
	Assume that $M$ and $M'$ have the same dimension, are compact, connected and that $M \neq \emptyset$. Suppose also that $M'$ is 1-connected. Then we have for every $J \in \mathcal{P}^{\varphi}, |J \cap I| =1$.
\end{lemma}
\begin{proof}[Proof of Lemma \ref{lemma_property_partition}]
	This follows from the proof of Lemma 47(i) in \cite{JZ21}\footnote{In this lemma, it is assumed that $\partial M' \neq \emptyset$. However, this is not necessary.}.
\end{proof}

\begin{lemma}[helicity inequality]\label{lemma_helicity_inequality}
	Let $k, n \in \N$ with $n \geq 2$. Let $M$ and $M'$ be compact $kn$-dimensional smooth manifolds, and  $\omega$ and $\omega'$ be maxipotent exact $k$-forms on $M$ and $M'$, respectively. Let $a$ be a positive real number and $\varphi: M \to M'$ a smooth orientation-preserving  embedding that intertwines $a\omega$ and $\omega'$. We denote by $O$ and $O'$ the orientations of $M$ and $M'$ induced by $\omega$ and $\omega'$. Then, for every $J \in \mathcal{P}^{\varphi}$ the following inequality holds:
	\begin{equation*}
		-a^n \sum_{i \in J \cap I}h_{M, O, \omega}(i) + \sum_{i' \in J \cap I'}h_{M', O', \omega'}(i') \geq 0.
	\end{equation*}
\end{lemma}

\begin{proof}
	This is Lemma 49 in \cite{JZ21}.
\end{proof}

\begin{lemma}\label{lemma_bdy_pt_mapped_to_interior}
	Let $X$ and $Y$ be topological manifolds, such that $X$ is compact and non-empty and $Y$ is connected. Let $\varphi:X \to Y$ be a continuous injective map, such that $\varphi(\partial X)\subseteq \partial Y$. Then $\varphi$ is surjective.
\end{lemma}
\begin{proof}
	Consider the equivalence relation $\sim$ on $X\sqcup X:=\left(\{0, 1\}\times X \right)$ given by $(i, x)\sim (i, x), \forall x\in X, i=0, 1$, and $(0, x)\sim(1, x), (1, x)\sim(0, x), \forall x \in \partial X$. We define $\widehat{X}$ to be the quotient \[X \sqcup X / \sim.\] We define $\widehat{Y}$ analogously. By \cite[Theorem 9.29]{Lee13}, $\widehat{X}$ and $\widehat{Y}$ are canonically topological manifolds without boundary. We denote by $\varphi \sqcup \varphi:X\sqcup X \to Y \sqcup Y$ the map induced by $\varphi$ on the disjoint union. We define
	\begin{align*}
		\widehat{\varphi}: \widehat{X} &\to \widehat{Y}\\
		[x] & \mapsto [(\varphi \sqcup \varphi)(x)].
	\end{align*}
	By our assumption that $\varphi(\partial X)\subseteq \partial Y$, this map is well-defined. Since $\varphi$ is continuous, $\widehat{\varphi}$ is continuous. By our assumption that $\varphi$ is continuous and injective and by Brouwer's invariance of domain theorem, it follows that $\widehat{\varphi}$ is injective. Since $\widehat{X}$ and $\widehat{Y}$ do not have boundary, it follows from the invariance of domain theorem that $\widehat{\varphi}(\widehat{X})$ is open in $\widehat{Y}$. This implies that $\varphi(X)$ is open in $Y$. Since $X$ is compact, $\varphi(X)$ is compact. Since $Y$ is Hausdorff, $\varphi(X)$ is closed in $Y$. Since $Y$ is connected and $X$ non-empty, it  follows that $\varphi(X)=Y$. This concludes the proof.
\end{proof}

\begin{proof}[Proof of Theorem \ref{thm_sets_fulfill_conditions}]\label{proof_thm_conditions}
	We show that $\mcat$ and $\mcattilde$ fulfill the conditions of Proposition \ref{prop_non_representable}. Conditions \ref{hyp_conn} and \ref{hyp_uncountable} hold by definition of $\mcat$. Condition \ref{hyp_no_rescaled_copy} holds because every element of $\mcat$ has volume 1, but no two elements of $\mcat$ are isomorphic.
	
	Condition \ref{hyp_small_sup} follows from:
	\begin{claim}\label{claim_embedding_cap_is_small}
		 For every $M\in\mcat, \wtilde{M}\in\mcattilde$, we have that
		\begin{equation*}
			c_{\wtilde{M}}(M)\leq a_0:=\sqrt[n]{\frac{3}{4}}.
		\end{equation*}
	\end{claim}
	\begin{proof}[Proof of Claim \ref{claim_embedding_cap_is_small}]
	Let $a\in(0, \infty)$ be such that there exists a morphism $\varphi: a\wtilde{M}\hookrightarrow M$.\footnote{If no such $a$ exists, then $c_{\wtilde{M}}(M)=0$ and the claim holds.}  We denote by $I$ and $\wtilde{I}$ the set of connected components of the boundary of $M$ and $\wtilde{M}$, respectively, by $i_0 \in I$ the boundary component of $M$ with helicity 4, by $\wtilde{i_0}\in \wtilde{I}$ the boundary component of $\wtilde{M}$ with positive helicity, and by $J_0$ the partition element containing $\wtilde{i_0}$.
	\begin{claim}\label{claim_exclude_partition}
		We have that
		\begin{enumerate}[label = (\roman*)]
		\item $i_0 \in J_0$, \label{claim_positve_with_positive}
		\item $J_0\setminus \{\wtilde{i_0}, i_0\}\neq\emptyset$. \label{claim_J_0_neg_element}
		\end{enumerate}
	\end{claim}
	We denote by $h: I \to \R$ and $\wtilde{h}: \wtilde{I} \to \R$ the boundary helicity maps for $M$ and $\wtilde{M}$ and by $I_-$ the set of boundary components of $M$ that have negative helicity. By Lemma \ref{lemma_property_partition} every partition element $J \in \mathcal{P}^{\varphi}$ (as defined in \eqref{eqn_partiton}) contains exactly one element of $\wtilde{I}$. Hence, Lemma \ref{lemma_helicity_inequality} with $J_0$ yields
	\begin{equation}\label{ineq_helicity_J_0}
		-a^n \wtilde{h}(\wtilde{i_0})+\sum_{i\in J_0\cap I} h(i)\geq 0.
	\end{equation}

	\begin{proof}[Proof of Claim \ref{claim_exclude_partition}]
		Since the first term on the left of \eqref{ineq_helicity_J_0} is negative, the sum on the left is positive, which implies that $i_0\in J_0$. This proves \ref{claim_positve_with_positive}.
		We denote by $\wtilde{I_-}$ the set of boundary components of $\wtilde{M}$ that have negative helicity. By our assumption on $\mcat$ and by definition of $\mcattilde$, we have that
		\begin{equation}\label{eqn_sum_neg_helicities}
		\sum_{i\in I_-}h(i)=-3 \leq \sum_{\wtilde{i}\in \wtilde{I_-}}\wtilde{h}(\wtilde{i}),
		\end{equation}
	 Since the volume of $\wtilde{M}$ is strictly bigger than that of $M$ and $a\wtilde{M} \hookrightarrow M$ via $\varphi$, we have that $a<1$. Combining this with equation \eqref{eqn_sum_neg_helicities}, it follows that
	\begin{equation}\label{ineq_neg_helicities}
		\sum_{i\in I_-}h(i) < a^n \sum_{\wtilde{i}\in \wtilde{I_-}}\wtilde{h}(\wtilde{i}).
	\end{equation}
	By using Lemma \ref{lemma_helicity_inequality} with every $J\in \partition{\varphi}\setminus \{J_0\}$, summing up the  inequalities, and using Claim \ref{claim_exclude_partition}\ref{claim_positve_with_positive}, we get
	\begin{equation*}
		-a^n \sum_{\wtilde{i}\in \wtilde{I_-}}\wtilde{h}(\wtilde{i})  + \sum_{i\in I_-\setminus\{J_0\}}h(i) \geq 0.
	\end{equation*}
	Combinig this with \eqref{ineq_neg_helicities}, it follows that
	\begin{equation*}
		\sum_{i\in I_-}h(i) < \sum_{i\in I_-\setminus\{J_0\}}h(i).
	\end{equation*}
	It follows that $J_0 \cap I_- \neq \emptyset$. Statement \ref{claim_J_0_neg_element} follows. This concludes the proof of Claim \ref{claim_exclude_partition}.
	\end{proof}
	
	By Lemma \ref{lemma_property_partition} applied with $\varphi: \wtilde{M}\to M$ and $J_0$ and the fact that $J_0\subseteq \wtilde{I}\sqcup I$, we have that $J_0\setminus \{\wtilde{i_0}\}=J_0\cap I$. It follows that
	\begin{align*}
		a^n &\leq \frac{h(i_0)+\sum_{i\in J_0\setminus\{i_0, \wtilde{i_0}\}}h(i)}{\wtilde{h}(\wtilde{i_0})}\tag{by \eqref{ineq_helicity_J_0} and Claim \ref{claim_exclude_partition}\ref{claim_positve_with_positive}}\\
		&\leq \frac{h(i_0)-1}{\wtilde{h}(\wtilde{i_0})} \tag{by Claim \ref{claim_exclude_partition}\ref{claim_J_0_neg_element} and the definition of $\mcat$}\\
		&\leq \frac{3}{4}\tag{since $\wtilde{h}(\wtilde{i_0})\geq 4 =h(i_0)$}
	\end{align*}
	The statement of Claim \ref{claim_embedding_cap_is_small} follows.
	\end{proof}
	
	We prove that condition \ref{hyp_nonsurjective_embedding} holds. Let $M\in\mcat$. Let $(Y, \Omega)$ be a connected object of $\Forms{kn}{k}$, such that $d\Omega=0$ and there exists a non-surjective morphism $\varphi: M \hookrightarrow Y$. By Lemma \ref{lemma_bdy_pt_mapped_to_interior}, we have that there exists a point $x_0\in \Int(Y)\cap \varphi(\partial M)$. Since $\varphi$ is a smooth embedding, there exists a smooth parametrization $\psi:\R^{kn} \to U$ for $\Int(Y)$ around $x_0$, such that $\Omega|_{U}$ is maxipotent, $\psi(0)=x_0$, $\psi(\R^{kn-1}\times\{0\})= U\cap\varphi(\partial M)$ and $\psi(\R^{kn-1}\times (-\infty, 0])=U\cap \varphi(M)$.
	
	For $m\in \N$, a point $p\in\R^m$ and $r>0$, we denote by $B_r^m(p)$ and $\overline{B}_r^m(p)$ the open and closed balls in $\R^m$ of radius $r$ around $p$. Let $\rho: B_1^{kn-1}(0)\to [0, \infty)$ be a smooth function with compact support, such that $\rho(0)>0$.
	
	We define
	\begin{align*}
		K:=\big\{(x_1, \dots, x_{kn})\in\R^{kn}\,\big|\, &(x_1, \dots, x_{kn-1})\in \overline{B}^{kn-1}_1(0)\text{ and }\\
		&0\leq x_{kn}\leq\rho(x_1, \dots, x_{kn-1})\big\}.
	\end{align*}
	The set $K$ is closed and bounded in $\R^{kn}$ and hence compact. We define
	\begin{equation}\label{eqn_def_M_with_bulge}
		\wtilde{M}:=\psi(K)\cup\varphi(M).
	\end{equation}
	This is a submanifold with boundary of $Y$. We equip $\wtilde{M}$ with the restriction of $\Omega$. By \cite[p.12, footnote 28]{JZ21}, there exists a manifold $\wtilde{M}_0 \in \Forms{kn}{k}_0$ which is isomorphic to $\wtilde{M}$.
	
	\begin{claim}\label{claim_M_0_tilde_in_mcattilde}
		We have that $\wtilde{M}_0 \in \mcattilde$.
	\end{claim}
	\begin{proof}
		We check the corresponding conditions for $\wtilde{M}$. Since both $\psi(K)$ and $\varphi(M)$ are compact, it follows that $\wtilde{M}$ is compact. 
		\begin{claim}\label{claim_1conn_exact}
			$\wtilde{M}$ is 1-connected and $\Omega|_{\wtilde{M}}$ is exact.
		\end{claim}
		\begin{proof}[Proof of Claim \ref{claim_1conn_exact}]
		We choose a smooth function $\eta: \R^{kn}\to [0, 1]$ that is equal to 1 on $K$ and vanishes outside an open neighborhood of $K$. We define
			\begin{align*}
			h:[0, 1] \times \R^{kn} &\to \R^{kn},\\
			(t,(x_1, \dots, x_{kn})) &\mapsto (x_1, \dots, x_{n-1}, x_{kn}-\eta(x_1, \dots, x_{kn})tx_{kn}).
		\end{align*}
		The map $h$ is smooth and maps $K$ to $\overline{B}^{kn-1}_1(0)$. We define
		\begin{align*}
			&f:[0,1] \times \wtilde{M} \to \wtilde{M}\\
			&f_t:=
			\begin{cases}
				\psi\circ h_t\circ \psi^{-1}&\text{on }U\\
				id &\text{otherwise}
			\end{cases}
		\end{align*}
			The map $f_1$ is a smooth homotopy equivalence between $\wtilde{M}$ and $M$. It follows that $\wtilde{M}$ has the same homotopy type as $M$, and hence is 1-connected.
			
			By Cartan's general magic formula (Proposition \ref{prop_general_Cartan}) and our assumption that $d\Omega=0$, there exists a smooth family $(\alpha_t)_{t\in [0, 1]}$ of $(k-1)$-forms on $\wtilde{M}$, such  that for every $t\in [0, 1]$, \[\frac{d}{dt}f_t^*\left(\Omega|_{\wtilde{M}}\right) = d\alpha_t.\]It follows that
			\begin{equation}\label{eqn_integration_general_cartan}
				f_1^*\left(\Omega|_{\wtilde{M}}\right)-f_0^*\left(\Omega|_{\wtilde{M}}\right) = \int_0^1d\alpha_t=d\int_{0}^{1}\alpha_t.
			\end{equation}
			By assumption, $\omega$ is exact on $M$. Therefore $\varphi_*\omega$ is exact on $\varphi(M)$. Since $\Omega|_{\varphi(M)}=\varphi_*\omega$ and $f_1(\wtilde{M})\subseteq \varphi(M)$, it follows that $f_1^*\left(\Omega|_{\wtilde{M}}\right)$ is exact. Using equation \eqref{eqn_integration_general_cartan}, it follows that $f_0^*\left(\Omega|_{\wtilde{M}}\right)$ is exact. Since $f_0=id$, we have $f_0^*\left(\Omega|_{\wtilde{M}}\right)=\Omega|_{\wtilde{M}}$. This proves Claim \ref{claim_1conn_exact}.
			\end{proof}
			
			By construction, $\Omega|_{\wtilde{M}}$ is maxipotent.
			
			We equip the boundary of $\wtilde{M}$ with the orientation induced by the form $\omega^{\wedge n}$ on $M$. We show that $\wtilde{M}$ has the right volume and boundary helicity. By construction, the volume of $\wtilde{M}$ is strictly greater than that of $M$, which is equal to 1. By Stokes' Theorem for helicity (Lemma \ref{lemma_stokes_helicity}), the helicity of the boundary component of $\wtilde{M}$ intersecting $\psi(K)$ has been increased by the difference between the volume of $\wtilde{M}$ and the volume of $M$, while the other boundary components of $\wtilde{M}$ have the same helicity as the corresponding boundary components of $M$. Hence, the positive boundary component has helicity at least 4 and the negative helicities sum up to at least -3.
			
			This proves Claim \ref{claim_M_0_tilde_in_mcattilde}.		
	\end{proof}
	By construction, $\wtilde{M}\subseteq Y$ and the inclusion is a morphism $(\wtilde{M}, \Omega|_{\wtilde{M}})\hookrightarrow (Y, \Omega)$. By Claim \ref{claim_M_0_tilde_in_mcattilde}, it follows that condition \ref{hyp_nonsurjective_embedding} holds. Thus, we may apply Proposition \ref{prop_non_representable}. It follows that $c_{\mcattilde}$ is not closedly target-representable. This concludes the proof of Theorem \ref{thm_sets_fulfill_conditions}.
\end{proof}

\begin{proof}[Proof of Proposition \ref{prop_non_representable}]\label{proof_prop_non_representable}
	Let $X$ be an object of $\Forms{m}{k}$, such that
	\begin{equation}\label{eqn_target_cap_finite}
	a_0:=\sup_{M\in\mcat}c^X(M) <\infty.\,\footnote{By Remark \ref{rmk_a_0}, we do not need to consider the case $a_0=\infty$.}
	\end{equation}
	We define
	\begin{equation}\label{eqn_epsilon}
	\varepsilon:= 1-\sup\left\{c_{\wtilde{M}}(M)\,\big|\, M\in\mcat, \wtilde{M}\in\mcattilde\right\}.
	\end{equation}
	By assumption \ref{hyp_small_sup}, we have $\varepsilon>0$. Let $M\in\mcat$. We have $c^X(M)\leq a_0$. Hence there exist $a_M\in(0, a_0+\varepsilon)$ and a morphism $\varphi_M: M\hookrightarrow a_MX$.
	
	Using hypothesis \ref{hyp_conn}, the map $M \mapsto \varphi_M$ induces a map
	\begin{align}
		\begin{split}\label{eqn_map_conn_comp}
			\mcat &\to \left\{\text{connected components of }X\right\}\\
			M &\mapsto \text{unique }Y \text{ containing } \varphi_M(M).
		\end{split}
	\end{align}
	Since $X$ is second countable, it has countably many connected components. Using hypothesis \ref{hyp_uncountable}, it follows that the map defined in \eqref{eqn_map_conn_comp} is not injective. Hence there exist $M, M'\in\mcat$ and a connected component $Y$ of $X$, such that $M\neq M'$ and $\varphi_M(M)$ and $\varphi_{M'}(M')$ are both included in $Y$.
	
	Assume $\varphi_{M'}$ is surjective. Since $M\neq M'$, hypothesis \ref{hyp_no_rescaled_copy} implies that $\varphi_{M'}^{-1}\circ\varphi_M$ is not an isomorphism from $a_MM$ to $ a_{M'}M'$. Since $\varphi_{M'}$ is an isomorphism, it follows that $\varphi_M$ is not an isomorphism, and hence is not surjective. It follows that $\varphi_M$ and $\varphi_{M'}$ are not both surjective.
	
	Without loss of generality, assume that $\varphi_M$ is not surjective. This map is a morphism $M\hookrightarrow a_M Y$. Hence by hypothesis \ref{hyp_nonsurjective_embedding}, there exists $\wtilde{M_0}\in\mcattilde$, such that $\wtilde{M_0}\hookrightarrow a_MY\subseteq a_MX$. It follows that
	\begin{equation}\label{eq_a_0}
		c^X(\wtilde{M_0})\leq a_M
	\end{equation}
	We also have
	\begin{equation}\label{eqn_embedding_cap_bigger1}
		c_{\mcattilde}(\wtilde{M_0})\geq c_{\wtilde{M_0}}(\wtilde{M_0})\geq 1.
	\end{equation}
	We compute
	\begin{align*}
		&c_{\mcattilde}(\wtilde{M_0})- c^X(\wtilde{M_0}) - \sup_{M\in\mcat}c_{\mcattilde}(M)+ \sup_{M\in\mcat}c^X(M)\\
		\geq &1 - a_M +\varepsilon-1+a_0\tag{by (\ref{eqn_embedding_cap_bigger1}, \ref{eq_a_0}, \ref{eqn_epsilon}, \ref{eqn_target_cap_finite})}\\
		>& 0 \tag{since $a_M \in (0, a_0+\varepsilon)$}.
	\end{align*}
	It follows that $c_{\mcattilde} \neq c^X$. This concludes the proof of Proposition \ref{prop_non_representable}.
\end{proof}

We now give the proofs of Propositions \ref{prop_disjoint_union_nonrepresentable} and \ref{prop_product_of_surfaces}. The following observation is important for both proofs.
\begin{remark}[image of a closed manifold under an embedding]\label{rmk_image_conn_comp}
	Let $M$ be a non-empty closed connected topological manifold and $N$ a topological manifold of the same dimension. Let $\varphi: M \hookrightarrow N$ be a topological embedding. The image of $\varphi$ is a connected component of $Y$. Indeed, by invariance of the domain, and using that $M$ has no boundary, we have that $\varphi(M)$ is open in $N$. Since $M$ is compact, $\varphi(M)$ is compact. Since $N$ is Hausdorff, it follows that $\varphi(M)$ is closed in $N$. Hence, since $M\neq\emptyset$, we have that $\varphi(M)$ is a connected component of $N$.
	\demo
\end{remark}

\begin{proof}[Proof of Proposition \ref{prop_disjoint_union_nonrepresentable}]\label{proof_disjoint_union}
	Let $X$ be an object of $\Forms{kn}{k}$, where we omit the differential form from the notation. We show that $c^X(aM\sqcup M)$ is infinite for some $a\in(0, \infty)$. We denote by $\pi_0(X)$ the set of (path-) connected components of $X$. We define the set
	\begin{equation*}
		S:=\left\{\left(a, b\right)\in(0, \infty)^2\,\Big|\, \frac{a}{b}M\hookrightarrow X, \frac{1}{b}M\hookrightarrow X\right\}.
	\end{equation*}
	\setcounter{claim}{0}
	\begin{claim}\label{claim_S_countable}
		The set $S$ is countable.
	\end{claim}
	\begin{proof}[Proof of Claim \ref{claim_S_countable}]
			We define the set
		\begin{equation*}
			S':= \left\{A \in (0, \infty)\,\big|\, \exists X_0\in\pi_0(X): AM \cong X_0\right\}
		\end{equation*}
		and the map
		\begin{align*}
			f: S &\to S' \times S',\\
			(a, b)&\mapsto \left(\frac{a}{b}, \frac{1}{b}\right).
		\end{align*}
		We show that the map $f$ is well-defined. Let $(a, b)\in S$.\footnote{If $S$ is empty, then Claim \ref{claim_S_countable} holds.} We choose a morphism $\varphi: \frac{a}{b}M \hookrightarrow X$. Since $M$ is non-empty, closed, and connected, by Remark \ref{rmk_image_conn_comp}, there exists a connected component $X_0$ of $X$, such that $\varphi\left(M\right)= X_0$. Hence $\varphi$ is an isomorphism between $\frac{a}{b}M$ and $X_0$. Similarly, there exists a connected component $X_1$ of $X$, such that $\frac{1}{b}M \cong X_1$. This implies that the map $f$ is well-defined.
		
		A straight-forward argument shows that $f$ is injective.
		
		Let $A \in S'$.\footnote{If $S'$ is empty, then $S$ is also empty and Claim \ref{claim_S_countable} holds.} We choose a connected component $X_0$ of $X$, such that $AM\cong X_0$. Since the form $\omega$ on $M$ is maxipotent, it follows that the form on $X$ restricted to $X_0$ is maxipotent. Using $AM\cong X_0$ again, we have that $A^n\vol(M)=\vol(X_0)$.\footnote{Recall that a maxipotent form induces a volume form. We denote by $\vol(M)$ the induced volume of $M$.} It follows that $S'$ is contained in the image of the map that sends each maxipotent connected component $X_0$ of $X$ to $\sqrt[n]{\frac{\vol(X_0)}{\vol(M)}}$. Since $X$ is a manifold, $\pi_0(X)$ is countable. It follows that $S'$ is countable. Since $f$ is injective,  it follows that $S$ is countable. This proves Claim \ref{claim_S_countable}.
	\end{proof}
	
	The set
	\begin{align*}
			&\left\{a\in (0, \infty)\,\Big|\,c^X(aM \sqcup M)<\infty \right\}\\
			=&\left\{a\in (0, \infty)\,\Big|\,\exists b\in (0, \infty): \frac{1}{b}\left(aM \sqcup M \right)\hookrightarrow X \right\}
	\end{align*}
	is included in the image of $S$ under the projection to the first component. Therefore, by Claim \ref{claim_S_countable}, it is countable. It follows that the complement of this set is uncountable, and hence non-empty. Hence $c^X(aM\sqcup M)$ is infinite for some $a\in (0, \infty)$. The statement of Proposition \ref{prop_disjoint_union_nonrepresentable} follows.
\end{proof}

\begin{proof}[Proof of Proposition \ref{prop_product_of_surfaces}]\label{proof_product}
	\setcounter{claim}{0}
	Let $X$ be an object of $\Forms{2k}{k}$. We denote by $\pi_0(X)$ the set of (path-) connected components of $X$. We define the set
	\begin{align*}
		S:=\left\{\left(a, b\right)\in(0, \infty)^2\,\Big|\,\frac{a}{b}M_0\times\frac{1}{b}M_1 \hookrightarrow X\right\}.
	\end{align*}
	\begin{claim}\label{claim_S_countable_products}
		The set $S$ is countable.
	\end{claim}
	\begin{proof}[Proof of Claim \ref{claim_S_countable_products}]
		By Remark \ref{rmk_image_conn_comp}, since $M_0$ and $M_1$ are non-empty, closed and connected, we have that
		\begin{equation}\label{eqn_isom_product}
			(a, b)\in S\implies\exists X_0\in\pi_0(X): \frac{a}{b}M_0\times\frac{1}{b}M_1 \cong X_0.
		\end{equation}
		Let $X_0$ be a connected component of $X$. We show that the set
		\begin{equation*}
			S_{X_0}:=\left\{(a, b)\in(0, \infty)^2\,\Big|\, \frac{a}{b}M_0\times\frac{1}{b}M_1 \cong X_0  \right\}
		\end{equation*}
		is countable. We define the function
		\begin{align*}
			f:=f_{X_0}: S_{X_0} &\to \R,\\
			(a, b) &\mapsto \frac{a}{b}.
		\end{align*}
		Assume that $S_{X_0}$ is non-empty.\footnote{Otherwise, $S_{X_0}$ is countable, which is what we want to show.} We choose $(a, b)\in S_{X_0}$. We have that $\frac{a}{b}M_0\times\frac{1}{b}M_1 \cong X_0$. Since $M_0$ and $M_1$ are maxipotent, it follows that the restriction of the form on $X$ to $X_0$ is maxipotent. We have that
\[b=\frac{f(a, b)}{\vol(X_0)}.\]
This implies that $f$ is injective. Moreover, we have that
		\begin{equation}\label{eqn_imf}
			\mathrm{im}(f)\subseteq \left(\frac{a}{b}\Z + \frac{1}{b}\Z\right).
		\end{equation}
		This follows from the fact that $\vol(M_i)=1$, for $i=0,1$, and from the following claim.
		\begin{claim}\label{claim_criterion_isom_product}
			If $a, b, a', b'\in(0, \infty)$ are such that $\frac{a}{b}M_0\times\frac{1}{b}M_1 \cong \frac{a'}{b'}M_0\times\frac{1}{b'}M_1$ then we have that
			\begin{equation*}
				\frac{a'}{b'} \in \frac{a}{b}\Z + \frac{1}{b}\Z.
			\end{equation*}
		\end{claim}
		\begin{proof}[Proof of Claim \ref{claim_criterion_isom_product}]
			We denote
			\begin{equation*}
				\omega:=\frac{a}{b}\sigma_0\oplus\frac{1}{b}\sigma_1 \text{ and } \omega':=\frac{a'}{b'}\sigma_0\oplus\frac{1}{b'}\sigma_1,
			\end{equation*}
			and by $\mathrm{pr}_i:M_0\times M_1 \to M_i$  the canonical projection, $i=0, 1$. We choose an isomorphism $\varphi: \frac{a'}{b'}M_0\times\frac{1}{b'}M_1 \to \frac{a}{b}M_0\times\frac{1}{b}M_1$ with respect to $\omega'$ and $\omega$. We fix a point $x_1\in M_1$ and define
			\begin{align*}
				\iota_0 : M_0 &\to M_0 \times M_1\\
				x &\mapsto (x, x_1).
			\end{align*}
			We have that
			\begin{align*}
				&\frac{a'}{b'}= \vol\left(\frac{a'}{b'}\sigma_0\right)=\vol(\iota_0^*\omega')=\vol(\iota_0^*\varphi^*\omega)\\
				=&\int_{M_0}\left(\frac{a}{b}(\mathrm{pr}_0\circ\varphi\circ\iota_0)^*\sigma_0+ \frac{1}{b}(\mathrm{pr}_1\circ\varphi\circ\iota_0)^*\sigma_1\right)\\
				=&\frac{a}{b}\mathrm{deg}(\mathrm{pr}_0\circ\varphi\circ\iota_0)\int_{M_0}\sigma_0+\frac{1}{b}\mathrm{deg}(\mathrm{pr}_1\circ\varphi\circ\iota_0)\int_{M_1}\sigma_1 \tag{Degree Thm}\\
				\in&\frac{a}{b}\Z+\frac{1}{b}\Z\tag{since $\sigma_0$ and $\sigma_1$ have volume 1}
			\end{align*}
			This proves Claim \ref{claim_criterion_isom_product}.
		\end{proof}
		By \eqref{eqn_imf} $f$ has a countable image. Since $f$ is injective, it follows that $S_{X_0}$ is countable, as desired. Since $\pi_0(X)$ is countable, it follows that $\bigcup_{X_0\in\pi_0(X)}S_{X_0}$ is countable. Using \eqref{eqn_isom_product}, it follows that $S$ is countable. This proves Claim \ref{claim_S_countable_products}.
	\end{proof}
	
	We have
	\begin{align*}
		&\left\{a \in (0, \infty)\,\big|\, c^X(aM_0\times M_1)<\infty\right\}\\
		=&\left\{a\in(0, \infty)\,\big|\, \exists b\in(0, \infty): (a, b)\in S\right\}.
	\end{align*}
	By Claim \ref{claim_S_countable_products}, this set is countable. Hence its complement is uncountable and thus non-empty. Hence $c^X$ is not finite on every product $(aM_0\times M_1), a\in (0, \infty)$. The statement of Proposition \ref{prop_product_of_surfaces} follows.
\end{proof}

\appendix
\section{General Cartan formula and Stokes' Theorem for helicity}\label{section_appendix}
We state the general Cartan formula that we used in the proof of Theorem \ref{thm_sets_fulfill_conditions}.

Let $I$ be an interval, $M, N$ smooth manifolds, $f: I\times M \to N$ a smooth function, $k \in \N$, and $\omega \in \Omega^k(N)$. For every $t \in I$, we define $\alpha^t \in \Omega^{k-1}(M)$ and $\beta^t \in \Omega^k(M)$ by
\begin{align*}
	\alpha^t_x(v_1, \dots, v_{k-1}) &:= \omega\big(\partial_t f_t(x), Df_t(x)v_1, \dots, Df_t(x)v_{k-1}\big),\\
	\beta^t_x(v_1, \dots, v_k) &:= d\omega\big(\partial_t f_t(x), Df_t(x)v_1, \dots, Df_t(x)v_k\big).\\
\end{align*}
\begin{prop}[General Cartan formula]\label{prop_general_Cartan}
	We have \[\frac{d}{dt}f_t^*\omega=\beta^t+d\alpha^t, \quad \forall t\in I.\]
\end{prop}
This result follows from an elementary argument involving the Leibniz rule.

We state and prove Stokes' theorem for helicity, which we used in Example \ref{exple_rounded_rectangle} and the proof of Theorem \ref{thm_sets_fulfill_conditions}. Let $k, n \in \mathbb{N}$, such that $n \geq 2$. Let $(M, O)$ be a compact, oriented manifold of dimension $kn$, and $\omega$ be an exact $k$-form on $M$. We denote by $\omega_{\partial M}$ the pullback of $\omega$ by the canonical inclusion of $\partial M$ into $M$ and by $O_{\partial M}$ the orientation of $\partial M$ induced by $\omega^{\wedge n}$.

\begin{lemma}[Stokes' Theorem for helicity]\label{lemma_stokes_helicity}
	The following equality holds:
	\begin{equation*}
		\int_{M, O} \omega^{\wedge n} =h(\partial M, O_{\partial M}, \omega_{\partial M}).
	\end{equation*}
\end{lemma}
\begin{proof}[Proof of Lemma \ref{lemma_stokes_helicity}]
	Let $\alpha$ be a primitive of $\omega$. Then, we have that $\omega^{\wedge n} = d\big(\alpha \wedge \omega^{\wedge (n-1)}\big)$. Using Stokes' Theorem, we get
	\begin{equation*}
		\int_{M, O}\omega^{\wedge n} = \int_{\partial M, O_{\partial M}} \alpha \wedge \omega^{\wedge (n-1)} = h(\partial M, O_{\partial M}, \omega_{\partial M}).
	\end{equation*}
	This completes the proof.
\end{proof}

\bibliographystyle{amsalpha}
\bibliography{mathematicalBibliography}

\end{document}